\documentclass[11pt]{amsart}
\usepackage{graphicx, color, url, amsmath}
\usepackage[all]{xy}
\usepackage[pdftex,lmargin=1in,rmargin=1in,tmargin=1in,bmargin=1in]{geometry}

\newtheorem{theorem}{Theorem}[section]
\newtheorem{corollary}[theorem]{Corollary}

\newtheorem{proposition}[theorem]{Proposition}

\theoremstyle{definition}
\newtheorem{definition}[theorem]{Definition}
\newtheorem{remark}[theorem]{Remark}

\renewcommand\th{^{\text{th}}}

\newcommand\Ta{\mathbb{T}_\alpha}
\newcommand\Tb{\mathbb{T}_\beta}

\newcommand\Xs{\mathbb{X}}
\newcommand\Os{\mathbb{O}}

\newcommand\Rect{\mathrm{Rect}}

\newcommand\Sym{\mathrm{Sym}}

\newcommand\zz {{\mathbb{Z}}}
\newcommand\rr {{\mathbb{R}}}

\newcommand\del {{\partial}}

\newcommand\sss {{\mathfrak{s}}}
\newcommand\C {{\mathcal{C}}}

\newcommand\spc {{\operatorname{Spin^c}}}

\newcommand\rk {{\operatorname{rank}}}

\newcommand\s{\mathbf s}
\newcommand\xs{\mathbf x}
\newcommand\ys{\mathbf y}

\newcommand\Chain {\mathfrak{A}}

\newcommand\alphas{\boldsymbol\alpha}
\newcommand\betas{\boldsymbol\beta}

\newcommand\EmptyRect{\Rect^\circ}

\newcommand\cx {\mathfrak{c}}

\newcommand\ff {{\mathbb{F}}}

\newcommand\Dom {{\mathcal{D}}}
\newcommand\Poly {{\mathcal{P}}}

\newcommand\Cone{{\operatorname{Cone}}}
\newcommand\CF {\mathit{CF}}
\newcommand\HF {\mathit{HF}}

\newcommand\Gens {\mathbf{S}}
\newcommand\HFm{\mathbf{HF}^-}

\newcommand\HFplus {\HF^+}
\newcommand \CFplus {\CF^+}

\newcommand \CFminus {\CF^-}
\newcommand \HFminus {\HF^-}

\newcommand \HFinf {\HF^{\infty}}

\newcommand\HFinfty {\HFinf}
\newcommand\HFred{{HF}_{\operatorname{red}}}

\newcommand\HFhat{\widehat{\HF}}
\newcommand\CFhat{\widehat{\CF}}

\newcommand\HFLminus{{HFL}^-}

\newcommand\HFLhat{\widehat{{HFL}}}
\newcommand\CFKhat{\widehat{{CFK}}}
\newcommand\CFKminus{{CFK}^-}
\newcommand\HFKhat{\widehat{{HFK}}}
\newcommand\HFKminus{{HFK}^-}

\title[Grid Diagrams in Heegaard Floer Theory]{Grid Diagrams in Heegaard Floer Theory}

\author[Ciprian Manolescu]{Ciprian Manolescu}
\thanks{The author was partially supported by NSF grant number DMS-1104406.}
\address {Department of Mathematics, UCLA, 520 Portola Plaza\\ 
Los Angeles, CA 90095, USA}
\email {cm@math.ucla.edu}
\begin{document}

\begin{abstract}
We review the use of grid diagrams in the development of Heegaard Floer theory. We describe the construction of the combinatorial link Floer complex, and the resulting algorithm for unknot detection. We also explain how grid diagrams can be used to show that the Heegaard Floer invariants of $3$-manifolds and $4$-manifolds are algorithmically computable (mod $2$).
\end{abstract}

\maketitle

\section{Introduction}

Invariants coming from gauge theory and symplectic geometry have played a major role in the study of low-dimensional manifolds. The topological applications of gauge theory started with the work of Donaldson, who used the Yang-Mills equations to get constraints on the intersection forms of smooth $4$-manifolds \cite{Don}.  Counting solutions to the anti-self-dual Yang-Mills equations  yielded invariants that were able to distinguish between homeomorphic, but not diffeomorphic, $4$-manifolds \cite{DonPol}. Floer \cite{Floer} used the same equations to construct an invariant of $3$-manifolds, which became known as instanton Floer homology.

In the 1990's came the advent of the Seiberg-Witten (or monopole) equations \cite{SeWi1, SeWi2, Witten} in four dimensions. These can replace the Yang-Mills equations for most applications, and have better compactness properties. The corresponding monopole Floer homology for $3$-manifolds was fully developed by Kronheimer and Mrowka in \cite{KMBook}; see also \cite{MarcolliWang, Spectrum, Froyshov}.

Ozsv\'ath and Szab\'o \cite{HolDisk, HolDiskTwo, HolDiskFour} developed Heegaard Floer theory as a more computable alternative to Seiberg-Witten theory. Instead of gauge theory, they used pseudo-holomorphic curve counts in symplectic manifolds to define invariants of $3$-manifolds, knots, links, and $4$-manifolds. In particular, their mixed Heegaard Floer invariants of $4$-manifolds are conjecturally the same as the Seiberg-Witten invariants, and can be used for the same applications---for example, to detect exotic smooth structures. In dimension $3$, the equivalence between Heegaard Floer theory and Seiberg-Witten theory has recently been established, thanks to work of Kutluhan-Lee-Taubes \cite{KLT1, KLT2, KLT3, KLT4, KLT5} and Colin-Ghiggini-Honda \cite{CGH1, CGH2, CGH3, CGH4}.

The definitions of the Yang-Mills, Seiberg-Witten, and Heegaard Floer invariants have in common the use of solution counts to nonlinear elliptic PDE's. (In the Heegaard Floer case, these are the  nonlinear Cauchy-Riemann equations, which define pseudo-holomorphic curves.) Consequently, the invariants above are fundamentally different from more traditional topological invariants such as homology and homotopy groups, Reidemeister torsion, etc. The latter are known to be algorithmically computable---their definitions do not involve analysis. However, the traditional invariants are insufficient to detect the subtle information about $3$-manifolds and $4$-manifolds that comes from gauge theory or symplectic geometry.

The purpose of this article is to survey recent advances in the field that resulted in combinatorial descriptions for most of the Heegaard Floer invariants. These advances started with an idea of Sarkar, who observed that pseudo-holomorphic curves can be counted explicitly if one uses a certain kind of underlying Heegaard diagram. In the case of knots and links in $S^3$, \emph{grid diagrams} can be successfully used for this purpose. This allowed the Heegaard Floer invariants of knots and links in $S^3$ to be described combinatorially \cite{MOS, MOST}. Further, $3$-manifolds and $4$-manifolds can be represented in terms of links in $S^3$ via surgery diagrams and Kirby diagrams, respectively. One can then show that the Heegaard Floer invariants (modulo $2$) of $3$- and $4$-manifolds are algorithmically computable, by expressing these invariants in terms of those of the corresponding link \cite{IntSurg, LinkSurg, MOT}. 

There are several alternate combinatorial approaches to computing some of the Heegaard Floer invariants. These approaches include nice diagrams \cite{SarkarWang, LMW, OSSu2},  cubes of resolutions \cite{CubeRes, BaldwinLevine}, and bordered Floer homology \cite{LOThat}. In some cases, the resulting algorithms are more efficient than the ones based on grid diagrams. Nevertheless, grid diagrams are the most encompassing method, and they are the focus of our survey.

\medskip
\textbf{Acknowledgements.}  I owe an intellectual debt to my collaborators Robert Lipshitz, Peter Ozsv\'ath, Sucharit Sarkar, Zolt\'an Szab\'o, Dylan Thurston, and Jiajun Wang, who all played an important role in the development of the subject. I would also like to thank Tye Lidman for many helpful expository suggestions.

\section{Heegaard Floer homology and related invariants}

This section contains a quick overview of Heegaard Floer theory.

The theory started with the work of Ozsv\'ath and Szab\'o, who in \cite{HolDisk, HolDiskTwo}  introduced a set of $3$-manifold invariants in the form of modules over the polynomial ring $\zz[U].$ Roughly, their construction goes as follows. We represent a closed, oriented $3$-manifold $Y$ by a \emph{marked Heegaard diagram}, consisting of the following data:
\begin {itemize}
\item $\Sigma$, a closed oriented surface of genus $g$;
\item $\alphas = \{\alpha_1, \dots, \alpha_g\}$, a collection of disjoint, homologically linearly independent, simple closed curves on $\Sigma$.  By attaching $g$ disks to $\Sigma$ along the curves $\alpha_i$, and then attaching a $3$-ball, we obtain a handlebody $U_{\alpha}$ with boundary $\Sigma$;
\item $\betas = \{\beta_1, \dots, \beta_g\}$ a similar collection of curves on $\Sigma$, specifying a handlebody $U_{\beta}$; 
\item a basepoint $z \in \Sigma - \cup \alpha_i - \cup \beta_i$; 
\end {itemize}
such that $Y = U_{\alpha} \cup_{\Sigma} U_{\beta}.$ Any $3$-manifold can be represented in this way; see Figure~\ref{fig:HF} for a schematic picture.

\begin{figure}
\begin{center}
\input{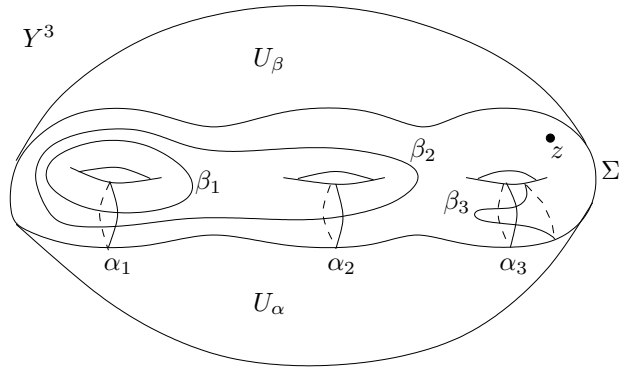}
\caption{A marked Heegaard diagram for $S^1 \times S^2$.}
\label{fig:HF}
\end{center}
\end{figure}

Next, we consider the tori
$$ \Ta = \alpha_1 \times \dots \times \alpha_g, \ \  \Tb = \beta_1 \times \dots \times \beta_g$$
inside the symmetric product $\Sym^g(\Sigma) = (\Sigma \times \dots \times \Sigma)/S_g$, which can be viewed as a symplectic manifold, cf. \cite{Perutz}. After small perturbations of the alpha or beta curves, we can assume that $\Ta$ and $\Tb$ intersect transversely.

To define the Heegaard Floer complex, we consider the intersection points
$$ \xs = \{x_1, \dots, x_g\} \in \Ta \cap \Tb$$ 
with $x_i \in \alpha_i \cap \beta_j \subset \Sigma$. Given two such intersection points $\xs, \ys$, a {\em Whitney disk} from $\xs$ to $\ys$ is defined to be a map $u$ from the unit disk $D^2$ to $\Sym^g(\Sigma)$, such that $u$ maps the lower half of the boundary $\del D^2$ to $\Ta$, the upper half to $\Tb$, and $u(-1)=\xs, u(1)=\ys$. The space of relative homotopy classes of Whitney disks from $\xs$ to $\ys$ is denoted $\pi_2(\xs, \ys)$. There is a natural map 
$$\Ta \cap \Tb \to \spc(Y), \ \ \ \xs \to \sss_z(\xs)$$ 
sending each intersection point to a $\spc$ structure on $Y$, such that $\pi_2(\xs, \ys)$ is empty when $\sss_z(\xs) \neq \sss_z(\ys).$

For every $\phi \in \pi_2(\xs, \ys)$, we can consider the moduli space of pseudo-holomorphic disks in the class $\phi$. These disks are solutions to the nonlinear Cauchy-Riemann equations, which depend on the choice of a family $J=\{J_t\}_{t \in [0,1]}$ of {\em almost complex structures} on the symmetric product. The moduli space has an expected dimension $\mu(\phi)$ (the Maslov index), and it comes with a natural action of $\rr$ by automorphisms of the domain. If $\mu(\phi)=1$, then for generic $J$, after dividing by the $\rr$ action the moduli space becomes just a finite set of points, which can be counted with certain signs. We let $c(\phi, J) \in \zz$ denote the resulting count. Moreover, we denote by $n_z(\phi)$ the intersection number between $\phi$ and the divisor $\{z\} \times \Sym^{g-1}(\Sigma)$. If $\phi$ admits any pseudo-holomorphic representatives, the principle of positivity of intersections for holomorphic maps implies that $n_z(\phi) \geq 0$.

Fix a $\spc$ structure $\sss$ on $Y$. Under a certain assumption on the Heegaard diagram (admissibility), one defines the Heegaard Floer complex $\CFminus(Y, \sss)$ to be the $\zz[U]$-module freely generated by intersection points $\xs \in \Ta \cap \Tb$ with $\sss_z(\xs)=\sss$. The differential on $\CFminus(Y, \sss)$ is given by the formula:
$$ \del \xs = \sum_{\{\ys \in \Ta \cap \Tb \mid \sss_z(\ys)=\sss\}} \sum_{\{\phi \in \pi_2(\xs, \ys) \mid \mu(\phi)=1 \}} c(\phi, J) \cdot U^{n_z(\phi)} \ys.$$ 

It can be shown that $\del^2 = 0$. Starting from the complex $\CFminus=\CFminus(Y, \sss)$ one defines the various versions of \emph {Heegaard Floer homology} to be:
$$ \HFminus = H_*(\CFminus), \ \  \HFplus = H_*( U^{-1} \CFminus / \CFminus ),$$
$$  \HFinfty = H_*(U^{-1}\CFminus),\ \ \HFhat = H_*(\CFminus/ (U=0)).$$

Although the complex $\CFminus$ depends on the choices of Heegaard diagram and almost complex structure, the homology groups do not:

\begin{theorem}[\cite{HolDisk}]
The Heegaard Floer homology modules $\HFminus(Y, \sss)$, $\HFplus(Y, \sss)$, $\HFinfty(Y, \sss)$, $\HFhat(Y, \sss)$ are invariants of the $3$-manifold $Y$ equipped with the $\spc$ structure $\sss$. 
\end {theorem}
 
If $\circ$ is any of the four flavors of Heegaard Floer homology ($-$, $+$, $\infty$, or $\wedge$), we will denote by $\HF^\circ(Y)$ the direct sum of $\HF^\circ(Y, \sss)$ over all $\spc$ structures $\sss$.
 
Heegaard Floer homology has found numerous applications to $3$-dimensional topology. Among them we mention: detection of the Thurston norm \cite{GenusBounds}, detection of whether the $3$-manifold fibers over the circle \cite{Ni3F}, and a characterization of which Seifert fibrations admit tight contact structures \cite{LiscaStipsicz}.

The reader may wonder about the differences between the different variants of Heegaard Floer homology. 
The full power of the theory comes from the versions $\HFminus$ and $\HFplus$, which contain (roughly) equivalent information. The version $\widehat{HF}$ is weaker: it  suffices for many $3$-dimensional applications, but it does not give any non-trivial information about closed $4$-manifolds. (As explained below, the mixed invariants of $4$-manifolds are constructed by combining the plus and minus theories.) The version $\HFinfty$ is the least useful, being determined by classical topological information, at least when we work modulo $2$ and we restrict to torsion $\spc$ structures; see \cite{Lidman}. 

As proved by Ozsv\'ath and Szab\'o in \cite{HolDiskFour}, the four variants of Heegaard Floer homology are functorial under $\spc$-decorated cobordisms. Precisely, given a $4$-dimensional manifold $W$ with $\del W = (-Y_0) \cup Y_1$, and a $\spc$ structure $\sss$ on $W$ with restrictions $\sss_0$ to $Y_0$ and $\sss_1$ to $Y_1$, there are induced maps:
$$ F^\circ_{W, \sss} : HF^\circ(Y_0, \sss_0) \longrightarrow HF^\circ(Y_1,  \sss_1),$$
where $\circ$ can stand for  any of the four flavors of Heegaard Floer homology. The maps $F^\circ_{W, \sss}$ are defined by counting \emph{pseudo-holomorphic triangles} in the symmetric product, with boundaries on three different tori.

Suppose we have a smooth, closed, oriented $4$-manifold $X$ with $b_2^+(X) > 1$. An {\em admissible cut} for $X$ is a smoothly embedded $3$-manifold $N \subset X$ which divides $X$ into two pieces $X_1$ and $X_2$ with $b_2^+(X_i) > 0$ for $i=1,2,$ and such that $\delta H^1(N; \zz) = 0 \subset H^2(X; \zz)$. 

An admissible cut can be found for any $X$. Given an admissible cut, one can delete a four-ball from the interior of each piece $X_i$ to obtain two cobordisms $W_1$ (from $S^3$ to $N$) and $W_2$ (from $N$ to $S^3$). Let $\sss$ be a $\spc$ structure on $X$. We consider the following diagram:
$$
\xymatrixcolsep{5pc}
\xymatrix{
 \HFminus(S^3)  \ar[r]^-{F^-_{W_1, \sss|_{W_1}}} \ar@{-->}[dr] & \HFminus(N, \sss|_N)& \\
& \HFred(N, \sss|_N) \ar@{^{(}->}[u] \ar@{-->}[dr] & \\
&  \HFplus(N, \sss|_N) \ar@{->>}[u] \ar[r]_-{F^+_{W_2, \sss|_{W_2}}}  & \HFplus(S^3),
 }
$$
where $\HFred$ is the image of a natural map $\HFplus \to \HFminus$. (This map exists for any $3$-manifold.) The conditions in the definition of an admissible cut ensure that the maps $F^-_{W_1, \sss|_{W_1}}$ and $F^+_{W_2, \sss|_{W_2}}$ factor through $\HFred$. The composition of the two lifts (indicated by dashed arrrows) is called a mixed map. The image of $1 \in \HFminus(S^3) \cong \zz$ under this map defines the \emph{Ozsv\'ath-Szab\'o mixed invariant} of the pair $(X, \sss)$. This is conjecturally equivalent to the well-known Seiberg-Witten invariant \cite{Witten}, and is known to share many of its properties. In particular, it can be used to distinguish homeomorphic $4$-manifolds that are not diffeomorphic. 

\section{Link Floer homology}

Ozsv\'ath-Szab\'o \cite{OSknots} and, independently, Rasmussen \cite{RasmussenThesis} used Heegaard Floer theory to define invariants for knots in $3$-manifolds: these are the various versions of \emph{knot Floer homology}.

Recall that a marked Heegaard diagram $(\Sigma, \alphas, \betas, z)$ represents a $3$-manifold $Y$. If one specifies another basepoint $w$ in the complement of the alpha and beta curves, this gives rise to a knot $K \subset Y$. Indeed, one can join $w$ to $z$ by a path in $\Sigma \setminus \cup \alpha_i$, and then push this path into the interior of the handlebody $U_{\alpha}$. Similarly, one can join $w$ to $z$ in the complement of the beta curves, and push the path into $U_{\beta}$. The union of these two paths is the knot $K$. For simplicity, we will assume that $K$ is null-homologous. (Of course, this happens automatically if $Y=S^3$.)

In the definition of the Heegaard Floer complex $\CFminus$ we kept track of intersections with $z$ through the exponent $n_z(\phi)$ of the variable $U$. Now that we have two basepoints, we have two quantities $n_z(\phi)$ and $n_w(\phi)$. One thing we can do is to count only disks in classes with $n_z(\phi)=0$, and keep track of $n_w(\phi)$ in the exponent of $U$. The result is a complex of $\zz[U]$-modules denoted $\CFKminus(Y, K)$, with homology $\HFKminus(Y, K)$. If we set the variable $U$ to zero, we get a complex $\CFKhat(Y, K)$, with homology $\HFKhat(Y, K)$. These are two of the variants of knot Floer homology. There exist many other variants, some of which involve classes $\phi$ with $n_z(\phi) \neq 0$ and $n_w(\phi) \neq 0$; an example of this, denoted $\Chain^{\pm}(K)$, will be mentioned in section~\ref{sec:surgery}.

Let us focus on the case when $Y=S^3$. The groups $\HFKhat(S^3, K)$ naturally split as direct sums:
$$ \HFKhat(S^3, K) = \bigoplus_{m, s \in \zz} \HFKhat_m(S^3, K, s).$$
Here, $m$ and $s$ are certain quantities called the Maslov and Alexander gradings, respectively. We can encode some of the information in $\HFKhat$ into a polynomial 
$$ \Poly_K(t, q) = \sum_{m, s \in \zz} t^m q^s \cdot \rk \ \HFKhat_m(S^3, K, s).$$

The specialization $\Poly_K(-1,q)$ is the classical Alexander polynomial of $K$. However, the applications of knot Floer homology go well beyond those of the Alexander polynomial. In particular, the genus of the knot, which is defined as
$$ g(K) = \min \{g \mid \exists\ \text{ embedded, oriented surface} \Sigma \subset S^3, \  \del \Sigma = K \},$$
can be read from $\HFKhat$:
\begin{theorem}[\cite{GenusBounds}]
\label {thm:genusK}
For any knot $K \subset S^3$, we have
$$ g(K) = \max \{s \geq 0 \mid \exists \ m, \ \widehat{HFK}_m(S^3, K, s) \neq 0 \}.$$
\end {theorem}
Since the only knot of genus zero is the unknot, we have:
\begin{corollary}
$K$ is the unknot if and only if $\Poly_K(q,t)=1$.
\end{corollary}

By a result of Ghiggini \cite{Ghiggini}, the polynomial $\Poly$ has enough information to also detect the right-handed trefoil, the left-handed trefoil, and the figure-eight knot. Ni \cite{Ni} extended the work in \cite{Ghiggini} to show that $S^3 \setminus K$ fibers over the circle if and only if $ \oplus_m \HFKhat_m(S^3, K, g(K)) \cong \zz$. Other applications of knot Floer homology include the construction of a concordance invariant called $\tau$ \cite{OStau}, and a complete characterization of which lens spaces can be obtained by surgery on knots \cite{Greene}.

If instead of a knot $K \subset S^3$ we have a link $L$ (a disjoint union of knots), we can define invariants $\HFLhat(S^3, L), \HFLminus(S^3, L)$, which are versions of  {\em link Floer homology}. When $L$ is a knot, $\HFLhat$ and $\HFLminus$ reduce to $\HFKhat$ and $\HFKminus$, respectively. In general, the definition of link Floer homology involves choosing a new kind of Heegaard diagram for $S^3$, in which the number of alpha (or beta) curves exceeds the genus of the Heegaard surface. The details can be found in \cite{Links}. If 
the diagram has $g+k-1$ alpha curves, it should also have $g+k-1$ beta curves, $k$ basepoints of type $z$, and $k$ basepoints of type $w$. In the simplest version, $k$ is the same as the number $\ell$ of components of the link, and joining the $w$ basepoints to the $z$ basepoints in pairs (by a total of $2\ell$ paths) produces the link $L$. Instead of $\zz[U]$, the link Floer complexes are defined over a polynomial ring $\zz/2[U_1, \dots, U_{\ell}]$, with one variable for each component. (In fact, we expect the complexes to be defined  over $\zz[U_1, \dots, U_{\ell}]$. However, at the moment some orientation issues are not yet settled, and the theory is only defined with mod $2$ coefficients.) 

More generally, we could have $k \geq \ell$, and break the link into more segments. We can then define a link Floer complex over $\zz/2[U_1, \dots, U_k]$, with one variable for each basepoint; see \cite{MOS, LinkSurg}. The homology of this complex is still $\HFLminus$, and all the variables corresponding to basepoints on the same link component act the same way. If we set one $U_i$ variable from each link component to zero in the complex, the resulting homology is $\HFLhat$. If we set all the $U_i$ variables to zero in the complex, the homology becomes 
\begin {equation}
\label {eq:vees}
\HFLhat(S^3, L) \otimes V^{k - \ell},
\end {equation}
where $V$ is a two-dimensional vector space over $\zz/2$.

\section{Grid diagrams and combinatorial link Floer complexes}

\begin {definition}
Let $L \subset S^3$ be a link. A {\em grid diagram} for $L$ consists of an $n$-by-$n$ grid in the plane with $O$ and $X$ markings inside, such that:
\begin {enumerate}
\item Each row and each column contains exactly one $X$ and one $O$;
\item As we trace the vertical and horizontal segments between $O$'s and $X$'s (with the vertical segments  passing over the horizontal segments), we see a planar diagram for the link $L$. 
\end {enumerate}
\end{definition}

An example is shown on the left hand side of Figure~\ref{fig:grid1}. It is not hard to see that every link admits a grid diagram. In fact, as a way of representing links, grid diagrams are equivalent to arc presentations, which originated in the work of Brunn \cite{Brunn}. The minimal number $n$ such that $L$ admits a grid diagram of size $n$ is called the {\em arc index} of $L$. 

\begin{figure}
\begin{center}
\includegraphics{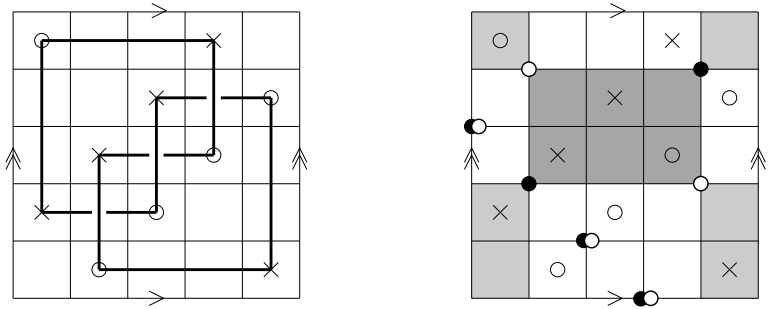}
\caption{A grid diagram for the trefoil, and two empty rectangles in $\EmptyRect(\xs, \ys)$. Here, $\xs$ is indicated by the collection of black dots, and $\ys$ by the collection of white dots. One empty rectangle is darkly shaded. The other rectangle is wrapped around the torus, and consists of the union of the four lightly shaded areas.}
\label{fig:grid1}
\end{center}
\end{figure}

Grid diagrams can be viewed as particular examples of Heegaard diagrams with multiple basepoints, of the kind discussed at the end of the previous section. Indeed, if we identify the opposite sides of a grid diagram $G$ to get a torus, we can let this torus be the Heegaard surface $\Sigma$, the horizontal circles be the $\alpha$ curves, the vertical circles be the $\beta$ curves, the $O$ markings be the $w$ basepoints, and the $X$ markings be the $z$ basepoints. A point $\xs = \{x_1, \dots, x_n\}$ in the intersection $\Ta \cap \Tb$ consists an $n$-tuple of points on the grid (one on each vertical and horizontal circle). There are $n!$ such intersection points, and they are precisely the generators of the link Floer complex. We denote the set of these generators by $\Gens(G)$.

\begin {definition}
Let $G$ be a grid diagram, and $\xs, \ys \in \Gens(G)$. We define a {\em rectangle} from $\xs$ to $\ys$ to be a rectangle on the grid torus with the lower left and upper right corner being points of $\xs$, the lower right and upper right corners being points of $\ys$, and such that all the other components of $\xs$ and $\ys$ coincide. (In particular, for such a rectangle to exist we need $\xs$ to differ from $\ys$ in exactly two rows.) A rectangle is called {\em empty} if it contains no components of $\xs$ or $\ys$ in its interior. The set of empty rectangles from $\xs$ to $\ys$ is denoted $\EmptyRect(\xs, \ys)$.
\end {definition}

Of course, the space $\EmptyRect(\xs, \ys)$ has at most two elements. An example where it has exactly two is shown on the right hand side of Figure~\ref{fig:grid1}. 

The reason why grid diagrams are useful in Heegaard Floer theory is that they make pseudo-holomorphic disks of Maslov index $1$ easy to count:

\begin{proposition}[\cite{MOS}]
\label{prop:MOS}
Let $G$ be a grid diagram, and let $\xs, \ys \in \Gens(G)$. Then, there is a $1$-to-$1$ correspondence:
$$ \Bigl \{ \phi \in \pi_2(\xs, \ys) \mid \mu(\phi)=1, \ c(\phi, J) \equiv 1(\operatorname{mod} \ 2) \text{ for generic } J \Bigr \}\longleftrightarrow  \EmptyRect(\xs, \ys).$$
\end{proposition}

\begin{proof}[Sketch of proof]
In any Heegaard diagram, if we have a relative homotopy class $\phi \in \pi_2(\xs, \ys)$, we can associate to it a two-chain $\Dom(\phi)$ on the Heegaard surface $\Sigma$, as follows. Let $n$ be the number of alpha (or beta) curves. Together, the alpha and the beta curves split $\Sigma$ into several connected regions $R_1, \dots, R_m$. For each $i$, let us pick a point $p_i$ in the interior of $R_i$, and define the {\em multiplicity} of $\phi$ at $R_i$ to be the intersection number $n_{p_i}(\phi)$ between $\phi$ and $\{p_i\} \times \Sym^{n-1}(\phi)$. We set
$$ \Dom(\phi) = \sum_i n_{p_i}(\phi) R_i.$$

This is called the {\em domain} of $\phi$. If $\phi$ admits any pseudo-holomorphic representatives, then the multiplicities $n_{p_i}(\phi)$ must be nonnegative.

Lipshitz \cite{LipshitzCyl} showed that the Maslov index of $\phi$ can be expressed in terms of the domain:
$$ \mu(\phi) = e(\Dom(\phi)) + N(\Dom(\phi)),$$
where $e$ and $N$ are certain quantities called the {\em Euler measure} and {\em total vertex multiplicity}, respectively. The Euler measure is additive on regions, that is, we can define $e(R_i)$ such that $e(\Dom(\phi)) = \sum_i n_{p_i}(\phi) e(R_i).$ If we take the sum $\sum_i e(R_i)$ we get the Euler characteristic of the Heegaard surface $\Sigma$. As for the total vertex multiplicity, it is the sum of $2n$ {\em vertex multiplicities} $N_q(\Dom(\phi))$, one for each point $q \in \xs$ or $q \in \ys$. The quantity $N_q(\Dom(\phi))$ is the average of the multiplicities of $\phi$ in the four quadrants around $q$.

In the case of a grid $G$, the regions $R_i$ are the $n^2$ unit squares of $G$. Each square has Euler measure zero. If we have $\phi \in \pi_2(\xs, \ys)$ with $\mu(\phi)=1$ and $c(\phi, J) \neq 0$, then the coefficients of $R_i$ in $\Dom(\phi)$ are nonnegative. This implies that $1=\mu(\phi)=N(\Dom(\phi))$ is 
a sum of vertex multiplicities $N_q(\Dom(\phi))$. Each $N_q(\Dom(\phi))$ is either zero or at least $1/4$. A short analysis shows that $\Dom(\phi)$ must be an empty rectangle.

Conversely, given an empty rectangle, there is a corresponding class $\phi$ with $\mu(\phi)=1$. An application of the Riemann mapping theorem shows that $\phi$ has an odd number of pseudo-holomorphic representatives for generic $J$.
\end {proof}

In view of Proposition~\ref{prop:MOS}, the link Floer complex associated to a grid can be defined in a purely combinatorial way. Precisely, we define $C^-(G)$ to be freely generated by $\Gens(G)$ over the ring $\zz/2[U_1, \dots, U_n],$ with differential:
$$ \del \xs = \sum_{\ys \in \Gens(G)} \sum_{\{r \in \EmptyRect(\xs, \ys) \mid \Xs_i(r) = 0,\ \forall i\}} U_1^{\Os_1(r)} \dots U_n^{\Os_n(r)} \cdot \ys.$$

Here, $\Os_i(r)$ encodes whether or not the $i\th$ marking of type $O$ is in the interior of the rectangle $r$: if it is, we set $\Os_i(r)$ to be $1$; otherwise it is $0$. The quantity $\Xs_i(r) \in \{0,1\}$ is defined similarly, in terms of the $i\th$ marking of type $X$. 

The homology of $C^-(G)$ is the link Floer homology $\HFLminus(S^3, L)$.

\begin {remark}
Although the complex $C^-(G)$ is defined with mod $2$ coefficients, one can add signs in the differential to get a complex over $\zz[U_1, \dots, U_n]$, whose homology is still a link invariant. See \cite{MOST, Gallais}.
\end{remark}

If in $C^-(G)$ we set one variable $U_i$ from each link component to zero, we get a complex $\widehat{C}(G)$ with homology $\HFLhat(S^3, L)$. Perhaps the simplest complex is 
$$\widetilde{C}(G) = C^-(G)/(U_1=U_2 = \dots = U_n =0),$$
for which we only need to count the empty rectangles with no markings of any type in their interior. The homology of $\widetilde{C}(G)$ is $\HFLhat(S^3, L) \otimes V^{n-\ell}$; compare \eqref{eq:vees}.

In particular, when $L=K$ is a knot, the homology of $\widetilde{C}(G)$ is $\HFKhat(S^3, K) \otimes V^{n-1}$.  
There exist simple combinatorial formulas for the Maslov and Alexander gradings of the generators in $\Gens(G)$, and from them one gets a bi-grading on $H_*(\widetilde{C}(G))$; see \cite{MOS}, \cite{MOST}. From here one can recover $\HFKhat(S^3, K)$ as a bi-graded group, taking into account that each $V$ factor is spanned by one generator in bi-degree $(0,0)$ and another in bi-degree $(-1,-1)$. This method of calculating $\HFKhat$ was implemented on the computer by Baldwin and Gillam \cite{BaldwinGillam}; see also Droz \cite{ Droz} for a more efficient program, using a variation of this method due to Beliakova \cite{Beliakova}. 

In view of Theorem~\ref{thm:genusK}, we see that grid diagrams yield an algorithm for detecting the genus of a knot. In particular, if we are given a knot diagram and want to see if it represents the unknot, we can turn it into a grid diagram (after some suitable isotopies), then set up the complex $\widetilde{C}(G)$, and check if $H_*(\widetilde{C}(G)) \cong V^{n-1}$. 

Among the other applications of the grid diagram method we mention one to contact geometry: the construction of new invariants for Legendrian and transverse knots in $S^3$ \cite{OST}. This led to numerous examples of transverse knots with the same self-linking number that are not transversely isotopic \cite{NOT}.

Another application of knot Floer homology via grid diagrams is Sarkar's combinatorial proof of the Milnor conjecture \cite{SarkarTau}. The Milnor conjecture states that the slice genus of the torus knot $T(p,q)$ is $(p-1)(q-1)/2$; a corollary is that the minimum number of crossing changes needed to turn $T(p,q)$ into the unknot is also $(p-1)(q-1)/2$. The conjecture was first proved by Kronheimer and Mrowka using gauge theory \cite{KMSlice}; for other proofs, see \cite{OStau}, \cite{RasmussenSlice}.

Slight variations of grid diagrams can be used to compute the knot Floer homology of knots inside lens spaces \cite{Hedden}, and of a knot inside its cyclic branched covers \cite{LevineCovers}.

Finally, we mention that there exists a purely combinatorial proof that $H_*(C^-(G))$ and $H_*(\widehat{C}(G))$ are link invariants \cite{MOST}.

\section{Three-manifolds and four-manifolds}
\label{sec:surgery}

For a general Heegaard diagram, counting pseudo-holomorphic disks in the symmetric product is very difficult. Why is it easy for a grid diagram? If we look at the proof of Proposition~\ref{prop:MOS}, a key point we find is that the regions $R_i$ have zero Euler measure. In fact, what is important is that they have {\em nonnegative} Euler measure: since the total vertex multiplicity is always nonnegative, the fact that $e(\Dom(\phi)) + N(\Dom(\phi)) = 1$ imposes tight constraints on the possibilities for $\Dom(\phi)$. 

In general, if a Heegaard surface $\Sigma$ can be partitioned into regions of nonnegative Euler measure, its Euler characteristic (which is the sum of all the Euler measures) must be nonnegative; that is, $\Sigma$ must be a sphere or a torus. Our grid diagrams were set on a torus. There is also a variant on the sphere, that produces another combinatorial link Floer complex, and in the end yields the same homology.

Instead of a knot in $S^3$, we could take a $3$-manifold $Y$ and try to compute its Heegaard Floer homology using this method. The problem is that a typical $3$-manifold does not admit a Heegaard diagram of genus $0$ or $1$; only $S^3$, $S^1 \times S^2$ and lens spaces do. However, Sarkar and Wang \cite{SarkarWang} proved that one can find a Heegaard diagram for $Y$, called a {\em nice diagram}, in which all regions except one have nonnegative Euler measure. (This is related to the fact that on a surface of higher genus we can move all negative curvature to a neighborhood of a point.) If we put the basepoint $z$ in the bad region (the one with negative Euler measure), then we can understand pseudo-holomorphic curve counts for all classes $\phi$ with $n_z(\phi)=0$. These are the only classes that appear when defining the complex $\CFhat(Y)$. Thus, we get an algorithm for computing $\widehat{HF}$ of any $3$-manifold. We refer to \cite{SarkarWang} for more details, and to \cite{OSSu2, OSSconvenient} for related work.

Similarly, one can compute the cobordism maps $\hat{F}_{W, \s}$ for any simply connected $W$ \cite{LMW}. These suffice to detect exotic smooth structures on some $4$-manifolds with boundary, but not on any closed $4$-manifolds.

This line of thought runs into major difficulties if one wants to understand combinatorially the plus and minus versions of $HF$, or the mixed invariants of $4$-manifolds. Instead, what is helpful is to reduce everything to the case of links in $S^3$, and then appeal to grid diagrams. This program was developed in \cite{LinkSurg, MOT}, and is summarized below.

Let us recall a theorem of Lickorish and Wallace \cite{Lickorish, Wallace}, which says that any closed $3$-manifold $Y$ is integral surgery on a link in $S^3$:
$$ Y = (S^3 \setminus \nu(L)) \cup_{\phi} (\nu(L)).$$
Here, $\nu(L)$ is a tubular neighborhood, and $\phi$ is a self-diffeomorphism of $\del \nu(L)$. The diffeomorphism can be specified in terms of a framing of the link, which in turn is determined by choosing one integer for each link component. For example, the Poincar\'e sphere is surgery on the right-handed trefoil with $+1$ framing. In general, we denote by $S^3_{\Lambda}(L)$ the result of surgery on $L$ with framing $\Lambda$.

Four-manifolds can also be expressed in terms of links, using Kirby diagrams \cite{Kirby}. By Morse theory, a closed $4$-manifold can be broken into a $0$-handle, some $1$-handles (represented in a Kirby diagram by circles marked with a dot), some $2$-handles (represented by framed knots), some $3$-handles, and a $4$-handle. The positions of the $1$-handles and $2$-handles determine the manifold. See Figure~\ref{fig:kirby} for a few examples.

\begin{figure}
\begin{center}
\input{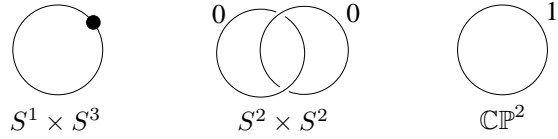}
\caption{Kirby diagrams for a few simple $4$-manifolds.}
\label{fig:kirby}
\end{center}
\end{figure}

The next step in the program is to express the Heegaard Floer homology of surgery on a  link in terms of data associated to the link. The first result in this direction was obtained by Ozsv\'ath and Szab\'o \cite{IntSurg}, who dealt with surgery on knots:

\begin {theorem}[\cite{IntSurg}]
\label{thm:IntSurg}
There is an (infinitely generated) version of the knot Floer complex, $\Chain^+(K)$, such that 
$$ \HFplus(S^3_n(K)) = H_*\bigl (\Cone\bigl( \Chain^+(K) \xrightarrow{\Phi^{K}_n} \Chain^+(\emptyset)\bigr) \bigr)$$
where in $\Chain^+(K), \Chain^+(\emptyset)$ we count pseudo-holomorphic bigons and in $\Phi^K_n$ we count pseudo-holomorphic triangles.

The complex $\Chain^+(\emptyset)$ is a direct sum of infinitely many copies of $\CFplus(S^3)$.  The inclusion of one of these copies into the mapping cone complex
$$\Cone\bigl(\Chain^+(K) \xrightarrow{\Phi^{K}_n} \Chain^+(\emptyset) \bigr)$$
induces on homology the map
$ F^-_{W,\sss} : \HFplus(S^3) \longrightarrow \HFplus(S^3_n(K))$
corresponding to the surgery cobordism $(2$-handle attachment along $K)$, equipped with a $\spc$ structure $\sss$.
\end {theorem}

The proof of Theorem~\ref{thm:IntSurg} is based on an important property of $\HFplus$ called the surgery exact triangle. The version $\HFminus$ does not have a similar exact triangle, but a slight variant of it, $\HFm$ does. The version $\HFm$ is obtained from $\HFminus$ by completion with respect to the $U$ variable. For torsion $\spc$ structures $\sss$, one can recover $\HFminus(Y, \sss)$ from $\HFm(Y, \sss)$, so in that case the two versions contain equivalent information. 

There is an analogue of Theorem~\ref{thm:IntSurg} with $\HFm$ instead of $\HFplus$, and with a knot Floer complex denoted $\Chain^-$ instead of $\Chain^+$. 

There is also an extension of Theorem~\ref{thm:IntSurg} to surgeries on links rather than single knots. Phrased in terms of $\HFm$, it reads:

\begin {theorem}[\cite{LinkSurg}]
\label{thm:LinkSurg}
If $L = K_1 \cup K_2 \subset S^3$ is a link with framing $\Lambda$, then $\HFm(S^3_{\Lambda}(L))$ is isomorphic to the homology of a  complex $\C^-(L, \Lambda)$ of the form
\begin{equation}
\label{eq:sc}
 \xymatrix{
 \Chain^-(L)\ar[d] \ar[r] \ar[dr] & \Chain^-(K_1) \ar[d] \\
 \Chain^-(K_2) \ar[r]  & \Chain^-(\emptyset)
 } 
 \end{equation}
where the edge maps count holomorphic triangles, and the diagonal map counts holomorphic quadrilaterals.

This can be generalized to links with any number of components. The higher diagonals involve counting higher holomorphic polygons. Further, the inclusion of the subcomplex corresponding to $L' \subseteq L$  corresponds to the cobordism maps given by surgery on $L - L'$. 
\end {theorem}

\begin{remark}
For technical reasons, at the moment Theorem~\ref{thm:LinkSurg} is only established with mod $2$ coefficients. 
\end{remark}

If $W$ is a cobordism between (connected) $3$-manifolds that consists of $2$-handles only, then we can express one boundary piece of $W$ as surgery on a link $L' \subset S^3$, and $W$ as a handle attachment along a link $L-L'$. Thus, Theorem~\ref{thm:LinkSurg} gives a description of the maps on $\HFm$ associated to any such cobordism $W$. In fact, $2$-handles are the main source of complexity in $4$-manifolds. Once we understand them, it is not hard to incorporate the maps induced by $1$-handles and $3$-handles into the picture. The result is a description of the Ozsv\'ath-Szab\'o mixed invariant of a $4$-manifold $X$ in terms of link Floer complexes. For this one needs to represent $X$ by a slight variant of a Kirby diagram, called a {\em cut link presentation}; we refer to \cite{LinkSurg} for more details.

\begin {theorem}[\cite{MOT}]
\label{thm:MOT}
Given any $3$-manifold $Y$ with a $\spc$ structure $\sss$, the Heegaard Floer homologies $\HFplus(Y, \sss)$ and $\HFm(Y, \sss)$ (with mod $2$ coefficients) are algorithmically computable. So are the mixed invariants $\Psi_{X, \sss}$ (mod $2$) for closed $4$-manifolds $X$ with $b_2^+(X) > 1$ and $\sss \in \spc(X)$. 
\end {theorem}

\begin{proof}[Sketch of proof]
We can represent the $3$-manifold or the $4$-manifold in terms of a link, as above (by a surgery diagram or a cut link presentation). The idea is then to take a grid diagram $G$ for the link, and apply Theorem~\ref{thm:LinkSurg}. We know that index $1$ holomorphic disks (bigons) on the symmetric product of the grid correspond to empty rectangles. However, to apply Theorem~\ref{thm:LinkSurg} we also need to be able to count higher pseudo-holomorphic polygons. In \cite{MOT}, it is shown that isolated pseudo-holomorphic triangles on the symmetric product are in $1$-to-$1$ correspondence with domains on  the grid of certain shapes, as shown in Figure~\ref{fig:ULeft}.

\begin{figure}
\begin{center}
\includegraphics{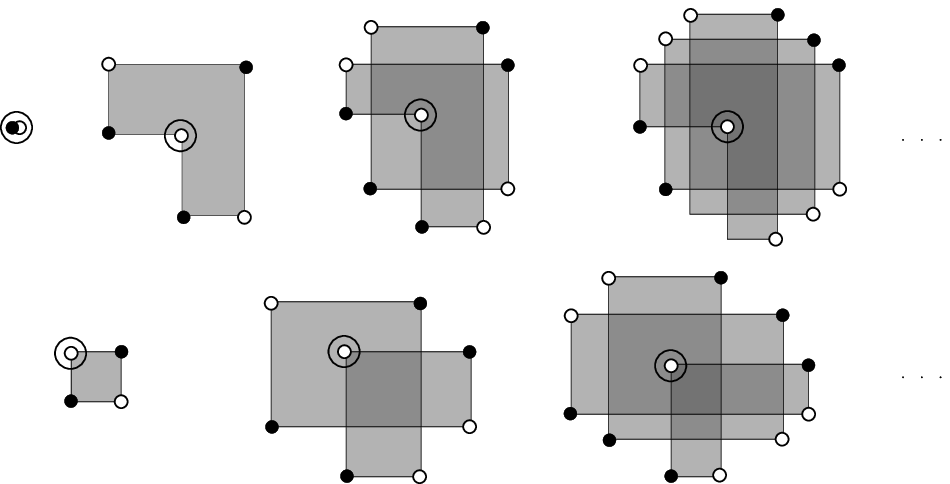}
\caption{Snail domains. Darker shading corresponds to
higher local multiplicities.  The domains in each row (top or bottom) are part of an infinite sequence, corresponding to increasing complexities. The larger circles represent certain fixed points on the grid, called  destabilization points. Each domain corresponds to a pseudo-holomorphic triangle in the symmetric product of the grid. }
\label{fig:ULeft}
\end{center}
\end{figure}

No such easy description is available for counts of pseudo-holomorphic $m$-gons on $\Sym^n(G)$ with $m \geq 4$. The trouble is that, unlike for $m=2$ or $3$, the counts for $m \geq 4$ depend on the choice of a generic family $J$ of almost complex structures on $\Sym^n(G)$. Still, the counts are required to satisfy certain constraints, coming from positivity of intersections and Gromov compactness. We define a {\em formal complex structure} on $G$ to be any count of domains on the grid that satisfies these constraints. 

A formal complex structure is a purely combinatorial object. Each such structure $\cx$ gives rise to a complex $\C^-(G, \Lambda, \cx)$, similar to \eqref{eq:sc}, but where instead of pseudo-holomorphic polygon counts we use the domain counts prescribed by $\cx$. In particular, a family of almost complex structures $J$ on the symmetric product produces a formal complex structure, whose corresponding complex is exactly \eqref{eq:sc}.  There is a definition of homotopy between formal complex structures, and if two such structures are homotopic, they give rise to quasi-isomorphic complexes $\C^-(G, \Lambda, \cx).$ 

We conjecture that any two formal complex structures on a grid diagram are homotopic. A weaker form of this conjecture, sufficient for our purposes, is proved in \cite{MOT}. Instead of an ordinary grid diagram $G$, we use its {\em sparse double} $G_\#$. This is obtained from $G$ by introducing $n$ additional rows, columns, and $O$ markings, interspersed between the previous rows and columns, as shown in Figure~\ref{fig:sparse}. The sparse double is not a grid diagram in the usual sense, because the new rows and columns have no $X$ markings. Nevertheless, it can still be viewed as a type of Heegaard diagram for the link, and pseudo-holomorphic bigons and triangles correspond to empty rectangles and snail domains, just as before. One result of \cite{MOT} is that on the sparse double, any two (sparse) formal complex structures are homotopic. 

\begin{figure}
\begin{center}
\input{sparse.pstex_t}
\caption{The sparse double of the grid diagram from Figure~\ref{fig:grid1}.}
\label{fig:sparse}
\end{center}
\end{figure}

With this in mind, the desired algorithm for computing $\HFm$ is as follows: Choose any formal complex structure on $G_\#$, and then calculate the homology of $\C^-(G, \Lambda, \cx).$ This homology is independent of $\cx$, so it agrees with the homology of the complex \eqref{eq:sc}. By Theorem~\ref{thm:LinkSurg}, this gives exactly $\HFm$ of surgery on the framed link. Similar algorithms can be constructed for computing $\HFplus$ and $\Phi_{X, \sss}$. 
\end {proof}

\section {Open problems}
\label {sec:open}

\subsection{Develop more efficient algorithms} 
A weakness of the grid diagram approach is that the size of the combinatorial knot Floer complex increases super-exponentially (like $n!$) with respect to the size of the grid. Nevertheless, in practice, computer programs \cite{BaldwinGillam, Droz} can calculate knot Floer homology (for knots and links in $S^3$) from diagrams of grid number up to $13$. The algorithms become much less effective for $3$-manifolds, and especially for $4$-manifolds:  this is because, for example, representing the $K3$ surface requires a grid of size at least $88$. 

A related open problem is to decide whether the unknotting problem can be solved in polynomial time.

\subsection{Combinatorial proofs.} To completely set the theory in elementary terms, it remains to give purely combinatorial proofs that the Heegaard Floer invariants are indeed invariants of the underlying object. For link Floer homology, this was achieved in \cite{MOST}. For $\HFhat$ of $3$-manifolds (defined from a class of diagrams called convenient, rather than from surgery formulas), a combinatorial proof of invariance appeared in \cite{OSSconvenient}. The cases of the other versions of $HF$, and of the mixed invariants of $4$-manifolds, remain open.

Also missing are combinatorial proofs for most of the topological properties of Heegaard Floer theory.
For example, it is not known how to prove combinatorially that knot Floer homology detects the genus of a knot.

\subsection{Loose ends} In terms of showing algorithmic computability, there are a few aspects of the theory that are not taken care of by Theorem~\ref{thm:MOT}:

\begin{itemize}
\item {\em Signs}. Extend the combinatorial descriptions to the invariants defined over $\zz$ (rather than over $\ff=\zz/2\zz$).

\item {\em Cobordism maps.} One can understand combinatorially the maps on Heegaard Floer homology induced by $2$-handle cobordisms, and the mixed map for closed $4$-manifolds, but not yet the maps induced by a general cobordism between $3$-manifolds (which may include $1$- and $3$-handles). 

\item {\em The uncompleted $\HFminus$ and $HF^{\infty}$.} Theorem~\ref{thm:MOT} is about $\HFm$ rather than $\HFminus$. For torsion $\spc$ structures, knowledge of $\HFm$ determines $\HFminus$. For nontorsion $\spc$ structures, understanding $\HFminus$ is basically equivalent to understanding $\HFm$ and $HF^\infty$; the latter group has not yet been computed.
\end{itemize}

\subsection{Other open problems in Heegaard Floer theory} As mentioned in the introduction, in dimension $3$, the Heegaard-Floer and Seiberg-Witten Floer homologies are known to be isomorphic. In dimension $4$, it is still open to prove that the mixed Ozsv\'ath-Szab\'o invariant of $4$-manifolds is the same as the Seiberg-Witten invariant.

Another important question is to understand the relationship of Heegaard Floer theory to Yang-Mills theory, and to the fundamental group of a $3$-manifold.

\bibliographystyle{amsalpha} 
\bibliography{biblio}

\end{document}